\documentclass[11pt]{amsart}
\usepackage{amssymb}
\usepackage{dsfont}

\newtheorem{theorem}{Theorem}[section]
\newtheorem{lemma}[theorem]{Lemma}
\newtheorem{proposition}[theorem]{Proposition}

\theoremstyle{definition}
\newtheorem{df}{Definition}
\newtheorem{example}[df]{Example}

\newcommand{\N}{\mathbb N}

\newcommand{\la}{\langle}
\newcommand{\ra}{\rangle}
\newcommand{\card}{\operatorname{card}}
\newcommand{\inte}{\operatorname{Int}}
\newcommand{\cl}{\operatorname{cl}}
\author{Micha{\l} Pop\l awski}
\address{Institute of Mathematics, \L \'od\'z University of Technology,
W\'olcza\'nska 215, 93-005 \L \'od\'z, Poland}
\email {michal.poplawski.m@gmail.com}
\title[Nontrivial convergent sequences]{The Baire category of the hyperspace \\ of nontrivial convergent sequences}
\date{}
\begin{document}
\begin{abstract}
Assume that $X$ is a regular space. We study topological properties of the family $S_c(X)$ of nontrivial convergent sequences in $X$ equipped with the Vietoris topology. In particular, we show that if $X$ has no isolated points, then $S_c(X)$ is a space of the first category which answers the question posed by S. Garcia-Ferreira and Y.F. Ortiz-Castillo.
\end{abstract}

\maketitle

\section{Introduction}

Let $X$ be a topological Hausdorff space. The Vietoris topology on the family $K(X)$ of all compact subsets of $X$ is generated by a base consisting of sets
\begin{equation} \la V_1,\ldots,V_n\ra:=\left\{K \in K(X) \colon K \subset \bigcup_{i=1}^n V_i \mbox{ and } K \cap V_i \neq \emptyset \mbox{ for } 1 \leq i \leq n \right\}, \label{bs} \end{equation}
where $n$ runs over $\N$ and $V_1,\ldots,V_n$ are nonempty open sets in $X$. 

Our notation is consistent with that used in \cite{GFOC}.
We say that $S \subset X$ is a nontrivial convergent sequence in $X$ if $S=\{x_n\}_{n \in \N} \cup \{\lim S\}$ for some injective sequence $(x_n)_{n \in \N}$ in $X$ which is convergent to some $\lim S \in X \setminus \{x_n\}_{n \in \N}.$ In other words, $S$ is a set of terms of an injective convergent sequence with its limit. Obviously, $S$ is compact for any space $X$, and $S_c(X)$ is empty for a discrete space $X$. Hence each closed subset $F$ of $S$ is compact. Consequently, the spaces $K(X)$, $CL(X)$, and the family of all closed subsets of $X$ with the topology generated by an analogous base, given by (\ref{bs}), introduce the same topology in their subspace $S_c(X)$.

%Note that, in an analogous way, we can define a family $S(X)$ of all nontrivial convergent nets, Namely, let $S \in S(X)$ if and only if $S=\{x_{\gamma}\}_{\gamma \in \Gamma} \cup \{\lim S\},$ where $\Gamma$ is a directed set and $(x_{\gamma})_{\gamma \in \Gamma}$ is an injective net convergent to some $\lim S \in X \setminus \{x_{\gamma}\}_{\gamma \in \Gamma}$. It is known that in Hausdorff spaces limits of nets are unique. Evidently $S(X)$ consists of closed sets.  Note that in general (even in metric spaces) a nontrivial convergent net $S$ does not have to be compact. This remark witnesses that the topology inherited from $CL(X)$ is natural in $S(X)$. For further information on nets and their applications in topology, see \cite{MR}, \cite{KL} or \cite{E}.%

In  general, separation axioms of the spaces $CL(X)$ and $K(X)$ depend on $X$.
%we can not expect a high numer of separation axioms for the spaces $CL(X)$ and $K(X)$. 
More precisely, $CL(X)$ is normal if and only if $X$ is compact (see \cite{VL}), and $K(X)$ is metrizable if and only if $X$ is metrizable (\cite{MC}).

The main aim of this paper is to show that $S_c(X)$ is of first category in itself under the assumption that $X$ is regular and crowded (i.e. has no isolated points). This result gives a positive answer to  Question 3.4 in \cite{GFOC}. The authors of \cite{GFOC} asked whether $S_c(X)$ is of the first category in itself if $X$ is a metric
crowded space. From now on, sets of the first category will be called also meager.
In our considerations we will use the Banach Category Theorem (see \cite{O}) which states that in any topological space the union of of a family of open meager sets is meager, too. Thanks to this fact, it suffices to construct a meager open neighbourhood of any $S \in S_c(X)$. Such a construction is possible in the case of regular crowded space $X$. It happened that no precise assumptions on a space $X$ were formulated in some theorems and proofs in \cite{GFOC}. Hence we have decided to repeat or modify the respective arguments from \cite{GFOC}, with an explicit evidence of properties of $X$, which would be helpful 
to understand all details.

\section{Main result}
We will follow some ideas taken from the paper \cite{GFOC} while considering some specific subsets of $S_c(X)$. For given $k,m \in \N, 1 \leq i \leq m$ and pairwise disjoint nonempty closed sets $C_1,\ldots,C_k$, we denote
$$N_{k}^{i}(m,\{C_j \colon j \leq k\}):=\left\{S \in S_c(X) \colon S \subset \bigcup_{j=1}^k C_j \mbox{ and } 1 \leq \card(S \cap C_j) \leq m \mbox{ for all } l \leq k, \  l \neq i\right\}$$
and $N_k(m,\{C_j \colon j \leq k\}):=\bigcup_{i=1}^k N_{k}^{i}(m,\{C_j \colon j \leq k\}).$

It can be immediately seen that, if $S \in N_{k}^{i}(m,\{C_j \colon j \leq k\}),$ then $\card(S \cap C_i)=\omega$.

Here we present a fact which can be derived directly from \cite[Lemma 3.1]{GFOC}.

\begin{lemma} \label{nw}
Let $X$ be a crowded topological space. Then the set $N_{k}^{i}(m,\{C_j \colon j \leq k\})$ is nowhere dense, whenever $k,m \in \N, 1 \leq i \leq m$ and $C_1,\ldots,C_k$ are pairwise disjoint, nonempty and closed sets.
\end{lemma}

% Clearly, as a result of this lemma, $N_k(m,\{C_j \colon j \leq k\})$ is nowhere dense as a finite union of nowhere dense sets.
Note that the assertion of the above fact can be lost if the set $N_k(m,\{C_j \colon j \leq k\})$ is considered as a subset of a closed subspace of $X$, which is not crowded.

\begin{example}
Consider sets $X:=[0,1], \ Y:=[0,1] \cup \{2\}, \ Z:= [0,1] \cup [2,3]$ with the Euclidean topology. Note that $X$ is a closed subspace of $Y$, $Y$ is a closed subspace of $Z$ and both spaces $X$ and $Z$ are crowded but $Y$ has one isolated point. Take $k:=2, \ m:=1, \ C_1:=[0,1], \ C_2:=\{2\}, \ i:=2$. By Lemma \ref{nw}, the set $N_{2}^{2}(1,\{C_1,C_2\})$ is nowhere dense in the space $S_c(Z)$. Nevertheless, this set is not nowhere dense in $S_c(Y)$. To check it, set $V_1:=[0,1], \ V_2:=\{2\}$, and consider the open set $\langle V_1,V_2\rangle$.  Obviously, each $S \in \langle V_1,V_2\rangle$ has exactly one point in $S \cap V_2$. Thus $S \in N_{2}^{2}(1,\{C_1,C_2\})$ and consequently, $\langle V_1,V_2\rangle \subset N_{2}^{2}(1,\{C_1,C_2\})$. Moreover, all sets $N_{k}^{i}(m,\{C_j \colon j \leq k\})$ as in Lemma \ref{nw} are nowhere dense in $S_c(X)$.
\end{example}

The next result (see \cite[Theorem 3.2]{GFOC}) is an application of the previous lemma.

\begin{lemma} \label{cat}
Suppose that $U_1,U_2$ are nonempty, closed and disjoint subsets of a crowded space $X$. Then $\langle\inte(U_1),\inte(U_2)\rangle$ is meager in $S_c(X)$.
\end{lemma} 
\begin{proof}
Thanks to Lemma \ref{nw} it suffices to observe that $\langle\inte(U_1),\inte(U_2)\rangle \subset \bigcup_{m \in \N} N_2(m,\{U_1,U_2\}).$ But this follows from the disjointness of closed sets $U_1,U_2$. 
\end{proof}

Now, we will construct the respective neighbourhoods of nontrivial convergent sequences.

\begin{lemma} \label{con}
Suppose $X$ is a Hausdorff space and $S=\{x_n\}_{n \in \N} \cup \{\lim S\} \in S_c(X)$. There are neighbourhoods $V_n, n \in \N$ of $x_{n}$'s and neighbourhood $V_S$ of $\lim S$ such that
$$V_1 \cap V_n=\emptyset \mbox{ for all n }\geq 2  \ \ \ \mbox{ and } V_1 \cap V_S=\emptyset.$$
\end{lemma}

\begin{proof}
Use the Hausdorff axiom to find disjoint neighbourhoods $V_1'$ of $x_1$, and $V_S$ of $\lim S$. Since $(x_n)_{n \in \N}$ is convergent to $\lim S$, the set $M:=\{m \geq 2 \colon x_m \notin V_S\}$ is finite. Again, for each $m \in M$ use the Hausdorff axiom to find disjoint neighbourhoods $V_{1,m}'$ of $x_1$, and $V_m$ of $x_m$. The intersection $V_1=V_1' \cap \bigcap_{m \in M} V_{1,m}$ satisfies 
$$V_1 \cap V_m=\emptyset \mbox{ for each } m \in M.$$
Then it suffices to define $V_k:=V_S$ for all $k \notin M \cup \{1\}.$ 
\end{proof}

\begin{proposition} \label{pr}
Suppose $X$ is a regular space and $S=\{x_n\}_{n \in \N} \cup \{\lim S\} \in S_c(X).$ Then there are nonempty, closed and disjoint sets $U_1,U_2$ with $S \in \langle\inte(U_1),\inte(U_2)\rangle.$
\end{proposition}

\begin{proof}
Let $V_S$, $V_n$, $n\in\N$, be the respective neighbourhoods of $\lim S$, $x_n$, $n \in \N$ considered in Lemma \ref{con}. Since $V_1 \cap (V_S \cup \bigcup_{n \geq 2} V_n)=\emptyset$, we have $V_1 \cap \cl(V_S \cup \bigcup_{n \geq 2})=\emptyset$. Then we use the regularity of $X$ to find a neighbourhood $W_1$ of $x_1$ such that $\cl(W_1) \subset V_1$. Put $U_1:=\cl(W_1), \ U_2:=\cl(V_S \cup \bigcup_{n \geq 2} V_n).$ Obviously, these sets are nonempty, closed and disjoint. We need to show that $S \in \langle\inte(U_1),\inte(U_2)\rangle.$ Indeed, $x_1 \in W_1 \subset \inte(U_1)$ and $\{\lim S\} \cup \bigcup_{n \geq 2} \{x_n\} \subset V_S \cup \bigcup_{n \geq 2} V_n \subset \inte(U_2).$
\end{proof}

\begin{theorem}
Suppose that $X$ is a regular crowded space. Then the space $S_c(X)$ is of the first category in itself.
\end{theorem}
\begin{proof}
Take $S \in S_c(X)$ and a neighbourhood of $S$ of the form $\langle\inte(U_1),\inte(U_2)\rangle$ considered in Proposition \ref{pr}. Then by Lemma \ref{cat}, this neighbourhood is meager. Therefore, by the Banach Category Theorem, $S_c(X)$ is of the first category as a union of meager neighbourhoods of its points.
\end{proof}

\end{document}